\documentclass[a4paper,11pt]{amsart}
\usepackage{amssymb}
\usepackage[a4paper,twoside,centering]{geometry}
\usepackage[arrow,matrix,curve]{xy}

\title{On cluster algebras from once punctured closed surfaces}
\author{Sefi Ladkani}
\address{%
Mathematical Institute of the University of Bonn \\
Endenicher Allee 60 \\
53115 Bonn, Germany}
\urladdr{http://www.math.uni-bonn.de/people/sefil}
\email{sefil@math.uni-bonn.de}

\thanks{The author is supported by DFG grant LA 2732/1-1 in the
framework of the priority program SPP 1388 ``Representation
theory''.}

\DeclareMathOperator{\add}{add}
\DeclareMathOperator{\Hom}{Hom}
\DeclareMathOperator{\ind}{ind}

\newcommand{\wt}{\widetilde}

\newcommand{\bQ}{\mathbb{Q}}
\newcommand{\bZ}{\mathbb{Z}}

\newcommand{\cA}{\mathcal{A}}
\newcommand{\cC}{\mathcal{C}}
\newcommand{\cP}{\mathcal{P}}
\newcommand{\cQ}{\mathcal{Q}}
\newcommand{\cS}{\mathcal{S}}
\newcommand{\cU}{\mathcal{U}}

\newtheorem*{theorem*}{Theorem}
\newtheorem{theorem}{Theorem}[section]
\newtheorem{prop}[theorem]{Proposition}
\newtheorem{lemma}[theorem]{Lemma}
\newtheorem{cor}[theorem]{Corollary}

\theoremstyle{definition}
\newtheorem{defn}[theorem]{Definition}
\newtheorem{remark}[theorem]{Remark}
\newtheorem{example}[theorem]{Example}

\numberwithin{equation}{section}

\begin{document}

\begin{abstract}
We show that many cluster-theoretic properties of the Markov quiver
hold also for adjacency quivers of triangulations of once-punctured
closed surfaces of arbitrary genus.

Along the way we consider the class $\cP$ of quivers introduced by
Kontsevich and Soibelman, characterize the mutation-finite
quivers that belong to that class and draw some conclusions
regarding non-degenerate potentials on them.
\end{abstract}

\maketitle

The markov quiver $Q$ shown below
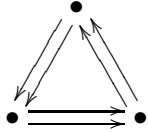
\[
\xymatrix@=1pc{
& {\bullet} \ar@<-0.5ex>[ddl] \ar@<0.5ex>[ddl] \\ \\
{\bullet} \ar@<-0.5ex>[rr] \ar@<0.5ex>[rr]
&& {\bullet} \ar@<-0.5ex>[uul] \ar@<0.5ex>[uul]
}
\]
has many intriguing properties:

\begin{itemize}
\item
There are exactly two arrows starting and two arrows
ending at any vertex of $Q$;

\item
The cluster algebra $\cA(Q)$ is not Noetherian \cite{Muller13};

\item
The upper cluster algebra $\cU(Q)$ is not equal to the cluster algebra
$\cA(Q)$~\cite{BFZ05,Muller13};

\item
$Q$ has no maximal green mutation sequences \cite{BDP13};

\item
There is a non-degenerate potential on $Q$ defining a $\Hom$-finite
cluster category having cluster-tilting objects that are not reachable
from the canonical one~\cite{Plamondon13};

\item
$Q$ does not belong to the class $\cP$ introduced by
Kontsevich and Soibelman~\cite{KS08}.
\end{itemize}

The purpose of this note is to demonstrate that this quiver does not
comprise a unique,
singular, example, but rather it is the simplest member of an infinite family of
mutation classes whose quivers share similar properties to the above.

Indeed, recall that $Q$ arises as the adjacency quiver
(in the sense of~\cite{FST08}) of any triangulation of a torus with
one puncture. It is therefore natural to consider closed surfaces of
higher genus. Adjacency quivers of triangulations of closed surfaces with
one puncture were studied in our previous work~\cite{Ladkani11} solving
a combinatorial problem of classifying the mutation classes of quivers with
constant number of arrows.
In another work~\cite{Ladkani12} we considered algebraic aspects of
triangulations of closed surfaces (with arbitrary number of punctures).
By using our previous results and adapting the existing proofs in the
literature we can show the following result.
\begin{theorem*}
Let $Q$ be an adjacency quiver of a triangulation of a once-punctured closed
surface. Then the following assertions hold:
\begin{enumerate}
\renewcommand{\theenumi}{\alph{enumi}}
\item \label{it:T:deg}
There are exactly two arrows starting and two arrows
ending at any vertex of $Q$;

\item \label{it:T:A}
The cluster algebra $\cA(Q)$ is not Noetherian;

\item \label{it:T:UA}
The upper cluster algebra $\cU(Q)$ is not equal to the cluster algebra
$\cA(Q)$;

\item \label{it:T:green}
$Q$ has no maximal green mutation sequences;

\item \label{it:T:reach}
There is a non-degenerate potential on $Q$ defining a $\Hom$-finite
cluster category having cluster-tilting objects that are not reachable
from the canonical one;

\item \label{it:T:P}
$Q$ does not belong to the class $\cP$.
\end{enumerate}
\end{theorem*}

Therefore each genus $g \geq 1$ gives rise to a finite mutation class
consisting of quivers with $6g-3$ vertices that satisfy the properties 
in the theorem.

Parts~\eqref{it:T:deg} and~\eqref{it:T:A} are shown in Section~\ref{sec:comb},
where we also recall some basic properties of the quivers considered in the
theorem. In Section~\ref{sec:UA} we show part~\eqref{it:T:UA}
by explicitly constructing elements in the upper cluster algebra that do not
belong to the cluster algebra,
whereas in Section~\ref{sec:green} we show parts~\eqref{it:T:green}
and~\eqref{it:T:reach}. Some properties of the class $\cP$ of quivers
(e.g.\ uniqueness, rigidity and finite-dimensionality of non-degenerate
potentials on them, see Theorem~\ref{t:P}) are discussed in
Section~\ref{sec:P}, where part~\eqref{it:T:P} of the theorem is shown as
a corollary of a stronger statement (Proposition~\ref{p:Qminus}).
Along the way we characterize the quivers belonging to the class $\cP$
whose mutation class is finite (Theorem~\ref{t:Pfinite}), and as a
consequence deduce an alternative proof of some of the results
in~\cite{GLS13}.

Unlike the once-punctured torus where the mutation class consists of a single
quiver, for surfaces of higher genus the mutation classes become very large.
A formula for the size of these classes in terms of the genus is given in the
paper~\cite{BacherVdovina02}. Explicit members in each class were
provided in~\cite{Ladkani11} (see in particular Figure~1 and Section~3.2
there). In Section~\ref{sec:constr} we present another construction that
produces different explicit members. For genus $2$ and $3$ the resulting
quivers are shown in Figure~\ref{fig:Qg0}.

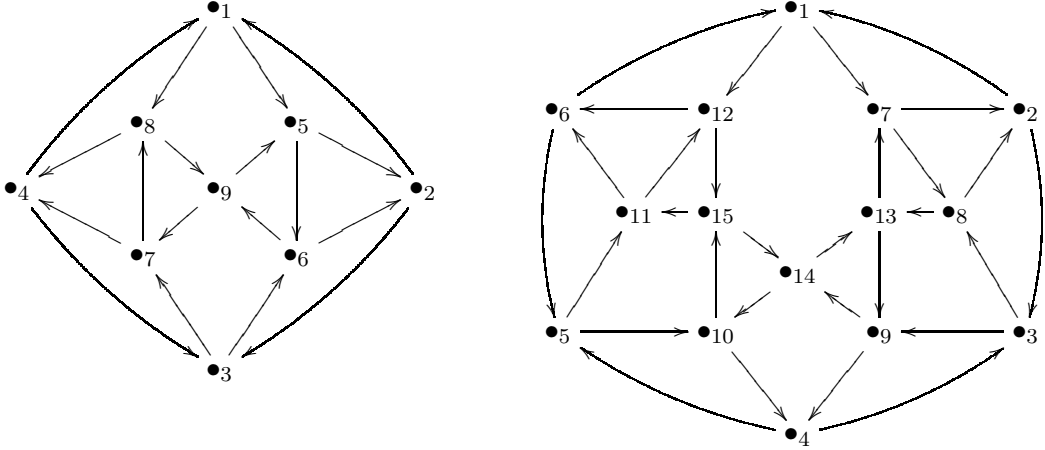
\begin{figure}
\begin{align*}
\xymatrix@=1pc{
&&& {\bullet_1} \ar[ddl] \ar[ddr] \\ \\
&& {\bullet_8} \ar[dll] \ar[dr]
&& {\bullet_5} \ar[dd] \ar[drr] \\
{\bullet_4} \ar@/^/[uuurrr] \ar@/_/[dddrrr]
&&& {\bullet_9} \ar[dl] \ar[ur]
&&& {\bullet_2} \ar@/_/[uuulll] \ar@/^/[dddlll] \\
&& {\bullet_7} \ar[uu] \ar[ull]
&& {\bullet_6} \ar[urr] \ar[ul] \\ \\
&&& {\bullet_3} \ar[uul] \ar[uur]
}
& \qquad &
\xymatrix@=0.8pc{
&&& {\bullet_1} \ar[ddl] \ar[ddr] \\ \\
{\bullet_6} \ar@/^/[uurrr] \ar@/_/[dddd]
&& {\bullet_{12}} \ar[ll] \ar[dd]
&& {\bullet_7} \ar[rr] \ar[ddr]
&& {\bullet_2} \ar@/_/[uulll] \ar@/^/[dddd] \\ \\
& {\bullet_{11}} \ar[uul] \ar[uur]
& {\bullet_{15}} \ar[l] \ar[dr]
&& {\bullet_{13}} \ar[dd] \ar[uu]
& {\bullet_8} \ar[l] \ar[uur] \\
&&& {\bullet_{14}} \ar[dl] \ar[ur] \\
{\bullet_5} \ar[uur] \ar[rr]
&& {\bullet_{10}} \ar[uu] \ar[ddr]
&& {\bullet_9} \ar[ul] \ar[ddl]
&& {\bullet_3} \ar[uul] \ar[ll] \\ \\
&&& {\bullet_4} \ar@/^/[uulll] \ar@/_/[uurrr]
}
\end{align*}
\caption{Adjacency quivers of triangulations of once-punctured
closed surfaces of genus $2$ (left) and genus $3$ (right).}
\label{fig:Qg0}
\end{figure}

\subsection*{Acknowledgement}
I would like to thank Arkady Berenstein for useful discussions.

\section{Combinatorial properties of the quivers}
\label{sec:comb}

In this section we explain parts~\eqref{it:T:deg} and~\eqref{it:T:A}
of the theorem.
For background on triangulations and their adjacency quivers we refer
to the paper~\cite{FST08} by Fomin, Shapiro and Thurston. In this
note we consider only ideal triangulations.
If $Q$ is a quiver, we denote by $Q_0$ its set of vertices and by $Q_1$
its set of arrows.

\begin{prop} \label{p:combmodel}
The set of the adjacency quivers of the (ideal) triangulations of a
closed surface of genus $g \geq 1$ with one puncture forms a mutation
class of quivers.
The following statements hold for any quiver $Q$ in this class:
\begin{enumerate}
\renewcommand{\theenumi}{\alph{enumi}}
\item \label{it:deg2}
$Q$ is connected without oriented cycles of length $2$, and
for any vertex $i$ of $Q$, there are exactly two arrows of $Q$ starting at $i$
and two arrows ending at~$i$.

\item
There are invertible maps $\phi, \psi \colon Q_1 \to Q_1$ with the following
properties:
\begin{enumerate}
\renewcommand{\theenumii}{\roman{enumii}}
\item
For any $\alpha \in Q_1$, the set $\{\phi(\alpha),\psi(\alpha)\}$ consists of
the two arrows starting at the vertex that $\alpha$ ends at;
\item
$\phi^3$ is the identity on $Q_1$.
\end{enumerate}

\item
For any arrow $\alpha \in Q_1$,
the path $\alpha \psi(\alpha) \psi^2(\alpha) \dots \psi^{12g-7}(\alpha)$ is an
Eulerian cycle in $Q$,
i.e.\ it is a cycle that traverses each arrow of $Q$ exactly once.
\end{enumerate}
\end{prop}
\begin{proof}
Recall from~\cite{FST08} the compatibility between flips
of triangulations and mutations of their adjacency quivers.
Now the claim in the preamble follows from the fact that any two
triangulations can be connected by a sequence of flips,
together with the observation that for a once-punctured closed surface,
any triangulation has no self-folded triangles and hence can be flipped
at any arc.

The other statements are special case of our results in~\cite[\S2.1]{Ladkani12}
dealing with adjacency quivers of triangulations of arbitrary closed surfaces.
The functions $\phi, \psi$ are denoted there $f, g$ (we changed the notation
to avoid confusion with the genus $g$).
Note that the technical condition~(T3) on the triangulation needed there is
automatically satisfied when there is only one puncture,
see~\cite[Lemma~5.3]{Ladkani12}.
\end{proof}

\begin{remark}
The mutation classes considered in the proposition were denoted by $\cQ_{g,0}$
in our previous work~\cite{Ladkani11}. In fact, it follows from the
main results of~\cite{Ladkani11} that the property expressed in
part~\eqref{it:deg2} of the proposition actually
characterizes these mutation classes.

More precisely,
if $Q$ is a connected quiver without oriented $2$-cycles such that
for any quiver $Q'$ in its mutation class and any vertex $i$ of $Q'$
there are exactly two arrows of $Q'$ starting at $i$ and two arrows ending
at $i$,
then $Q$ arises as the adjacency quiver of a triangulation of a once-punctured
closed surface.
\end{remark}

Now let $Q$ be an adjacency quiver of a triangulation of a once-punctured
closed surface.
From part~\eqref{it:deg2} of Proposition~\ref{p:combmodel} we deduce that
the cluster
exchange relations in the cluster algebra $\cA(Q)$ are homogeneous of degree
$2$ in the cluster variables.
Hence we can repeat the argument of Muller~\cite[\S11.2]{Muller13} and
deduce the following statement.

\begin{prop} \label{p:grad}
Let $Q$ be an adjacency quiver of a triangulation of a closed surface 
with one puncture.
\begin{enumerate}
\renewcommand{\theenumi}{\alph{enumi}}
\item
$\cA(Q)$ admits a non-negative grading such that all cluster variables
have degree~$1$.

\item
$\cA(Q)$ is infinitely generated and non-Noetherian.
\end{enumerate}
\end{prop}

\section{On the upper cluster algebra}
\label{sec:UA}

In this section we prove part~\eqref{it:T:UA} of the theorem.
We follow an idea communicated to me by
Berenstein~\cite{BerensteinRetakh} on angles in order to explicitly
construct an element in the upper cluster algebra that does not belong
to the cluster algebra.

We start by recalling, in the language of ice quivers~\cite[\S4]{Keller12},
the notion of cluster algebra of geometric type with skew-symmetric exchange
matrix~\cite{BFZ05,FominZelevinsky02,FominZelevinsky07}.
An \emph{ice quiver} is a quiver $\wt{Q}$ without oriented cycles of
length $2$ (hence also without loops) together with a subset of 
vertices called \emph{frozen}. We shall assume that the set of
vertices is $\wt{Q}_0 = \{1,2,\dots,m\}$, and the subset of frozen ones
is $\{n+1,\dots,m\}$ for some $n \leq m$.
A \emph{seed} $(\mathbf{x},\wt{Q})$ is a pair consisting of an ice quiver
$\wt{Q}$ together with a tuple $\mathbf{x}=(x_1,x_2,\dots,x_m)$ of
elements in the field $\bQ(u_1,u_2,\dots,u_m)$ which freely generate it.

Let $1 \leq k \leq n$ (i.e.\ $k$ is not frozen).
The \emph{mutation} of the seed $(\mathbf{x},\wt{Q})$ at $k$ is the seed
$(\mathbf{x'},\wt{Q}')$ where $\wt{Q}'=\mu_k(\wt{Q})$ is the quiver
obtained from $\wt{Q}$ by mutation at $k$ and
$\mathbf{x'}=(x'_1,\dots,x'_m)$ is defined by setting $x'_i=x_i$ for
$i \neq k$ and
\begin{equation} \label{e:seedmut}
x'_k \cdot x_k =
\prod_{\substack{\text{arrows} \\ k \to i}} x_i
+ \prod_{\substack{\text{arrows} \\ j \to k}} x_j
\end{equation}
where the arrows are considered in the quiver $\wt{Q}$. The frozen vertices
of $\wt{Q}'$ are the same as those in $\wt{Q}$.

Let $\wt{Q}$ be an ice quiver and $\mathbf{x}=(x_1,\dots,x_m)$ a sequence
of $m$ indeterminates. The \emph{upper cluster algebra} $\cU(\wt{Q})$
consists of the elements in $\bQ(x_1,\dots,x_m)$ that can be written as
Laurent polynomials with integer coefficients
in the elements of any cluster $\mathbf{x'}$ appearing
in a seed $(\mathbf{x'},\wt{Q}')$ that can be obtained from the seed
$(\mathbf{x},\wt{Q})$ by a sequence of mutations at non-frozen vertices.
Define the algebra $\cA(\wt{Q})$ to be the $\bZ$-subalgebra of
$\bQ(x_1,\dots,x_m)$ generated by the elements of these clusters.
It is a cluster algebra of geometric type according to the definition
in~\cite[\S 5]{FominZelevinsky02},
however in later references~\cite{BFZ05,FominZelevinsky07} the
\emph{cluster algebra} is defined as
$\cA(\wt{Q})[x_{n+1}^{-1},\dots,x_m^{-1}]$. 
Obviously, the two definitions coincide when there are no frozen vertices
(``no coefficients'').

We work in a slightly more general setting than is actually needed.
Let $(S,M)$ be a marked bordered oriented surface and let $n$ be the
number of arcs in any triangulation of $(S,M)$.
If $S$ is not closed,
in this section it would be convenient to think of the boundary segments
of $(S,M)$, which are sides of triangles in any triangulation of $(S,M)$,
as arcs labeled $n+1,n+2,\dots,m$ for some $m > n$.

For a triangulation $T$ of $(S,M)$ denote by $\wt{Q}_T$ the \emph{extended
adjacency quiver} of $T$ defined similarly to the ordinary adjacency quiver
$Q_T$ but taking into account also the boundary segments.
We think of $\wt{Q}_T$ as an ice quiver on $m$ vertices
labeled $1,2,\dots,m$ where the vertices corresponding to the boundary
segments, labeled $n+1,\dots,m$, are frozen and the full subquiver on the
non-frozen vertices is $Q_T$.

By abuse of notation, a \emph{seed} $(\mathbf{x},T)$ is a pair consisting
of a triangulation $T$ of $(S,M)$
together with a tuple $x=(x_1,x_2,\dots,x_m)$ of elements in the field
$\bQ(u_1,u_2,\dots,u_m)$ which freely generate it.
Let $1 \leq k \leq n$ and assume that the arc labeled $k$ can be flipped.
The \emph{mutation} of the seed $(\mathbf{x},T)$ is the seed
$\mu_k(\mathbf{x},T)=(\mathbf{x'},\mu_k(T))$
where $\mu_k(T)$ is the triangulation obtained
from $T$ by flipping the arc labeled $k$ and
$\mathbf{x'}=(x'_1,\dots,x'_m)$ is defined according to the usual rule
of seed mutation~\eqref{e:seedmut},
where the arrows are considered in the quiver~$\wt{Q}_T$.

The next lemma shows that the compatibility between flips of triangulations
and mutations of their adjacency quivers~\cite{FST08} holds also for
extended adjacency quivers.
\begin{lemma} \label{l:flipext}
$\wt{Q}_{\mu_k(T)} = \mu_k(\wt{Q}_T)$, hence the map defined by
$(\mathbf{x},T) \mapsto (\mathbf{x},\wt{Q}_T)$ is compatible with seed
mutations.
\end{lemma}
\begin{proof}
Consider the marked surface $(\wt{S},\wt{M})$ obtained from $(S,M)$ by gluing
to $S$ a small triangle near each boundary segment and adding as marked point
its inner vertex, see Figure~\ref{fig:SMtilde}.
Topologically, $S$ and $\wt{S}$ have the same genus and the same number of
boundary components, and the number of points in $\wt{M}$ on each boundary
component is twice as that of~$M$.
By construction, the arcs of
a triangulation $T$ of $(S,M)$ together with the boundary segments
naturally form a triangulation $\wt{T}$ of $(\wt{S},\wt{M})$, 
the map $T \to \wt{T}$ is compatible with flips of arcs of $T$,
and the extended adjacency quiver of $T$ is just the adjacency quiver of
$\wt{T}$.
\end{proof}

\begin{figure}
\begin{align*}
\xymatrix@=1pc{
{\cdot} \ar@{-}@<0.2ex>@/^/[rr] \ar@{.}@<-0.2ex>@/^/[rr]
\ar@{.}@<0.2ex>@/_/[dd] \ar@{-}@<-0.2ex>@/_/[dd]_{\gamma} 
&& {} \\ \\
{\cdot} \ar@{.}@<0.2ex>@/_/[rr] \ar@{-}@<-0.2ex>@/_/[rr]
&& {}
}
&&
\xymatrix@=1pc{
{\cdot} \ar@{-}@<0.2ex>@/^/[rr] \ar@{.}@<-0.2ex>@/^/[rr]
\ar@{-}@/_/[dd]_{\gamma}
&& {} \\
& {\cdot_p} \ar@{-}@/_/@<0.1pc>[ul] \ar@{.}@/_/@<-0.1pc>[ul]
\ar@{-}@<-0.1pc>@/^/[dl] \ar@{.}@<0.1pc>@/^/[dl] \\
{\cdot} \ar@{.}@<0.2ex>@/_/[rr] \ar@{-}@<-0.2ex>@/_/[rr]
&& {}
}
\end{align*}
\caption{Gluing a triangle and adding a marked point $p$ for a boundary
segment $\gamma$.}
\label{fig:SMtilde}
\end{figure}
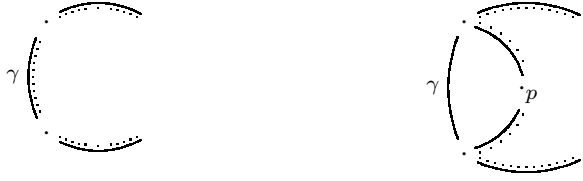

\begin{defn} \label{def:angles}
Let $(\mathbf{x},T)$ be a seed such that $T$ has no self-folded triangles.
Consider a triangle of $T$ with sides labeled $i,j,k$ 
and a marked point $p \in M$ as in the following picture
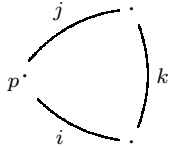
\[
\xymatrix@=1pc{
&& {\cdot} \ar@{-}@/^/[dd]^k \\
{{_p}\cdot} \ar@{-}@/^/[urr]^j \ar@{-}@/_/[drr]_i \\
&& {\cdot}
}
\]
We define the (commutative) \emph{angle} from $i$ to $j$ as
\begin{equation} \label{e:angle}
\measuredangle_p (i,j) =\frac{x_k}{x_i x_j}
\end{equation}
(in the notation we suppress the dependency of this quantity on the
seed $(\mathbf{x},T)$).

Let $p \in M$ be a marked point
and let $i_1,i_2,\dots,i_r$ be the arcs incident to $p$
arranged in a counterclockwise order. 
If $p$ lies on the boundary of $S$ then $i_1,i_r$
(which might coincide) are boundary
segments and the \emph{sum of angles at $p$} is defined as
\[
\measuredangle_p(\mathbf{x},T) =
\measuredangle_p (i_1,i_2) + 
\measuredangle_p (i_2,i_3) + \dots +
\measuredangle_p (i_{r-1},i_r) ,
\]
whereas if $p$ is a puncture, we define the sum of angles to be
\[
\measuredangle_p(\mathbf{x},T) =
\measuredangle_p (i_1,i_2) + 
\measuredangle_p (i_2,i_3) + \dots +
\measuredangle_p (i_{r-1},i_r) + \measuredangle_p(i_r,i_1) .
\]
\end{defn}

\begin{prop}[Preservation of angles] \label{p:angles}
Let $(\mathbf{x},T)$ be a seed such that $T$ has no self-folded triangles.
Assume that in the quiver $\wt{Q}_T$
there are exactly two arrows starting at the vertex $k$
and two arrows ending there.
Then $\mu_k(T)$ does not have self-folded triangles
and $\measuredangle_p(\mathbf{x},T) =
\measuredangle_p \mu_k(\mathbf{x},T)$ for any $p \in M$.
\end{prop}
\begin{proof}
It suffices to consider the contributions of angles inside the
quadrilateral where the flip of $k$ takes place. This quadrilateral looks
like the left picture below
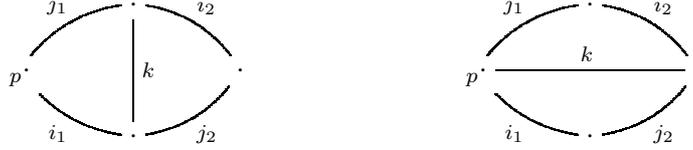
\begin{align*}
\xymatrix@=1pc{
&& {\cdot} \ar@{-}[dd]^k \\
{{_p}\cdot} \ar@{-}@/^/[urr]^{j_1} \ar@{-}@/_/[drr]_{i_1}
&&&& {\cdot} \ar@{-}@/_/[ull]_{i_2} \ar@{-}@/^/[dll]^{j_2} \\
&& {\cdot}
}
&&
\xymatrix@=1pc{
&& {\cdot} \\
{{_p}\cdot} \ar@{-}@/^/[urr]^{j_1} \ar@{-}@/_/[drr]_{i_1}
\ar@{-}[rrrr]^{k}
&&&& {\cdot} \ar@{-}@/_/[ull]_{i_2} \ar@{-}@/^/[dll]^{j_2} \\
&& {\cdot}
}
\end{align*}
where the indices $i_1,i_2,j_1,j_2$ of the side arcs satisfy
$\{i_1,i_2\} \cap \{j_1,j_2\} = \emptyset$ due to our assumption on
the vertex $k$ in $\wt{Q}_T$.

The effect of flip is shown in the right picture, so that 
$\wt{Q}_{\mu_k(T)}$ is equal to the mutation $\mu_k(\wt{Q}_T)$,
in this quiver the in-degree and out-degree of the vertex $k$
are equal to $2$ and no self-folded triangles have been created in
$\mu_k(T)$.

By symmetry it is enough to consider 
the contribution of angles at a marked point $p$ as shown.
Indeed, dividing the cluster exchange relation
\[
x_k x'_k = x_{i_1} x_{i_2} + x_{j_1} x_{j_2}
\]
by $x_{i_1} x_{j_1} x'_k$ we deduce that
\[
\measuredangle_p(i,j) = 
\frac{x_k}{x_{i_1} x_{j_1}} = \frac{x_{i_2}}{x_{j_1} x'_k} +
\frac{x_{j_2}}{x_{i_1} x'_k}
= \measuredangle_p(k,j_1) + \measuredangle_p(i_1,k) .
\]
\end{proof}

Now let $(S,M)$ be a marked surface and assume that either:
\begin{itemize}
\item
All points of $M$ lie on the boundary of $S$ (``no punctures''); or
\item
$S$ is closed and $|M|=1$.
\end{itemize}
In this case, any triangulation of $(S,M)$ does not have self-folded
triangles and the in-degree and out-degree of any non-frozen vertex
$1 \leq k \leq n$ in its extended adjacency quiver are both equal to $2$.
This has several consequences for a triangulation $T$:
\begin{enumerate}
\renewcommand{\labelenumi}{\theenumi.}
\item
Proposition~\ref{p:angles} applies.

\item
If $\mathbf{x'}$ is a cluster in any (usual) seed $(\mathbf{x'},\wt{Q}')$
mutable from $(\mathbf{x},\wt{Q}_T)$ at non-frozen vertices, there
exists a triangulation $T'$ of $(S,M)$ such that $(\mathbf{x'},T')$
is a seed.

\item
The exchange relations~\eqref{e:seedmut}
are homogeneous of degree $2$, hence the algebra
$\cA(\wt{Q}_T)$ admits a non-negative grading such that the elements of
each cluster have degree $1$.
\end{enumerate}

\begin{prop} \label{p:mu}
Let $p \in M$ be a marked point and $(\mathbf{x},T)$ be any seed.
\begin{enumerate}
\renewcommand{\theenumi}{\alph{enumi}}
\item \label{it:muLaurent}
$\measuredangle_p(\mathbf{x},T)$ is a Laurent polynomial in $x_1,\dots,x_m$;

\item \label{it:muT}
$\measuredangle_p(\mathbf{x},T)$ is invariant under seed mutations;

\item \label{it:muU}
$\measuredangle_p(\mathbf{x},T) \in \cU(\wt{Q}_T)$;

\item \label{it:muA}
$\measuredangle_p(\mathbf{x},T)$ is homogeneous of degree $-1$, 
hence it does not belong to $\cA(\wt{Q}_T)$.
\end{enumerate}
\end{prop}
\begin{proof}
Claims~\eqref{it:muLaurent} and~\eqref{it:muA}
are immediate from Definition~\ref{def:angles},
claim~\eqref{it:muT} follows from Proposition~\ref{p:angles} and
claim~\eqref{it:muU} follows from~\eqref{it:muLaurent} and~\eqref{it:muT}.
\end{proof}

By setting the variables $x_{n+1},\dots,x_m$ (corresponding to the boundary
segments) to~$1$, we obtain elements in the upper cluster algebra of
$\cU(Q_T)$ (without coefficients). We illustrate this in two examples.

\begin{example}
Consider a triangulation of a disc with $6$ marked points (i.e.\ a hexagon)
shown in the left picture.
Its extended adjacency quiver is shown in the right picture, where we
indicated the frozen vertices by $\square$.
On each arc we write the variable and
at each marked point we write the corresponding sum of angles,
which is a member of the upper cluster algebra $\cU(A_3)$.
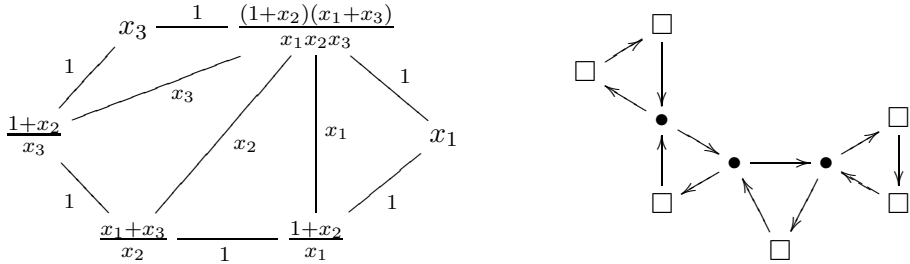
\begin{align*}
\xymatrix@=0.5pc{
& {x_3} \ar@{-}[rr]^(0.35){1}
&& {\frac{(1+x_2)(x_1+x_3)}{x_1 x_2 x_3}} \ar@{-}[ddr]^(0.6){1}
\ar@{-}[dddd]^{x_1} \ar@{-}[ddddll]^{x_2} \ar@{-}[ddlll]^{x_3} \\ \\
{\frac{1+x_2}{x_3}} \ar@{-}[uur]^1 &&&& {x_1} \ar@{-}[ddl]^1 \\ \\
& {\frac{x_1+x_3}{x_2}} \ar@{-}[uul]^1 && {\frac{1+x_2}{x_1}} \ar@{-}[ll]^1
}
&&
\xymatrix@=0.3pc{
&& {\square} \ar[dd] \\
{\square} \ar[urr] \\
&& {\bullet} \ar[drr] \ar[ull] && && &&
{\square} \ar[dd] \\
&& && {\bullet} \ar[rr] \ar[dll]
&& {\bullet} \ar[urr] \ar[ddl] \\
&& {\square} \ar[uu] && && &&
{\square} \ar[ull] \\
&& && & {\square} \ar[uul]
}
\end{align*}
In this case the upper cluster algebra and the cluster algebra coincide
as follows from~\cite{BFZ05}.
\end{example}

\begin{figure}
\begin{align*}
\xymatrix@=0.5pc{
&&& {\bullet_3} \ar[ddlll] \ar[dddd] \\ \\
{\bullet_1} \ar@<0.5ex>[rr] \ar@<-0.5ex>[rr]
&& {\bullet_2} \ar[uur] \ar[ddr] \\ \\
&&& {\bullet_4} \ar[uulll]
}
&&
\xymatrix{
{\cdot} \ar@{-}[rr]^{x_1} \ar@{-}@/_1.8pc/[ddrr]^{x_4}
\ar@{-}@/^1.7pc/[ddrr]_{x_3} \ar@{-}@(r,d)[]^{1}
&& {\cdot} \ar@{-}[dd]^{x_2} \\ \\
{\cdot} \ar@{-}[uu]^{x_2} && {\cdot} \ar@{-}[ll]^{x_1}
}
\end{align*}
\caption{The quiver $Q_{1,1}$ together with a corresponding triangulation
of the torus with one boundary component and one marked point.}
\label{fig:torus}
\end{figure}
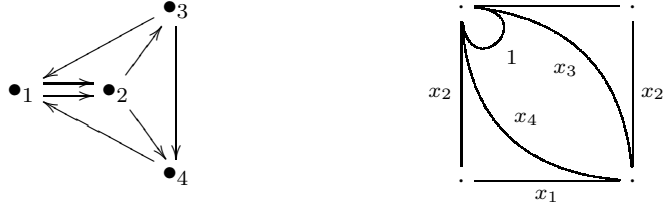

\begin{example} \label{ex:torus}
The quiver $Q_{1,1}$ shown in Figure~\ref{fig:torus}
is the adjacency quiver of any triangulation of the torus with one boundary
component and one marked point.
For the triangulation shown in Figure~\ref{fig:torus}, 
the element of Proposition~\ref{p:mu} (after specializing $x_5=1$) is given by
\[
\frac{x_1^2+x_2^2+x_3^2}{x_1 x_2 x_3} + \frac{x_1^2+x_2^2+x_4^2}{x_1 x_2 x_4}
+ \frac{1 + x_3^2 + x_4^2}{x_3 x_4} \in \cU(Q_{1,1}) .
\]
\end{example}

Now let $S$ be a closed surface and $|M|=1$.
In this case there are no frozen vertices (i.e.\ $\wt{Q}_T = Q_T$ for any
triangulation).
Since there is only one puncture, we may omit the reference to that puncture
and rewrite the sum of angles as
\begin{equation} \label{e:muT}
\measuredangle(\mathbf{x},T) = 
\sum_{\text{triangles $\{i,j,k\}$}}
\left(\frac{x_k}{x_i x_j} + \frac{x_i}{x_j x_k} + \frac{x_j}{x_k x_i}
\right)
= \sum \frac{x_i^2 + x_j^2 + x_k^2}{x_i x_j x_k} .
\end{equation}

\begin{cor}
Let $Q$ be an adjacency quiver of a triangulation of a once-punctured closed
surface. Then $\cA(Q) \neq \cU(Q)$.
\end{cor}

\begin{proof}
We adapt the proofs in~\cite{BFZ05,Muller13}.

Consider the sum $\mu = \measuredangle(\mathbf{x},T)$ of all the angles in
the triangles of a triangulation $T$ whose adjacency quiver is $Q$.
By Proposition~\ref{p:mu}, $\mu \in \cU(Q)$ but $\mu \not \in \cA(Q)$.
\end{proof}

\begin{remark}
In the case of the torus, the element $\mu$ is twice the element
$\frac{x^2+y^2+z^2}{xyz}$ considered in~\cite{Muller13}.
\end{remark}

\begin{example}
Consider the triangulation of a once-punctured surface of genus $2$ shown in
Figure~\ref{fig:Tg0}. Its adjacency quiver is shown in Figure~\ref{fig:Qg0}
and the element $\mu$ is
\[
\frac{x_1^2 + x_2^2 + x_5^2}{x_1 x_2 x_5} +
\frac{x_2^2 + x_3^2 + x_6^2}{x_2 x_3 x_6} +
\frac{x_3^2 + x_4^2 + x_7^2}{x_3 x_4 x_7} +
\frac{x_4^2 + x_1^2 + x_8^2}{x_4 x_1 x_8} +
\frac{x_5^2 + x_6^2 + x_9^2}{x_5 x_6 x_9} +
\frac{x_7^2 + x_8^2 + x_9^2}{x_7 x_8 x_9} .
\]
\end{example}

\section{On maximal green sequences}
\label{sec:green}

In this section we prove parts~\eqref{it:T:green} and~\eqref{it:T:reach}
of the theorem. We will not define maximal green mutation sequences
here, instead we refer the reader to the surveys~\cite{Keller11,Keller13}
by Keller and to the
paper~\cite{BDP13} by Br\"{u}stle, Dupont and P\'{e}rotin
initiating the study of such sequences.
For the basic notions on quivers with potentials that will be needed, we
refer to the paper~\cite{DWZ08} by Derksen, Weyman and Zelevinsky.

Throughout this section we fix a quiver $Q$ which is
an adjacency quiver of a triangulation of a once-punctured closed
surface of some genus $g \geq 1$. Let $\phi,\psi$ be as in
Proposition~\ref{p:combmodel}. Define an equivalence relation $\sim$ on the set
of arrows $Q_1$ by setting $\alpha \sim \alpha'$ if $\alpha'=\phi^r(\alpha)$
for some integer $r$
(in fact it suffices to consider $r \in \{0,1,2\}$ as $\phi^3$ is the
identity on $Q_1$).

Consider the following two potentials on $Q$:
\begin{align*}
W_0 &= \sum_{\alpha \in Q_1/\sim} \alpha \phi(\alpha) \phi^2(\alpha) \\
W_1 &= \sum_{\alpha \in Q_1/\sim} \alpha \phi(\alpha) \phi^2(\alpha)
     - \beta \psi(\beta) \psi^2(\beta) \dots \psi^{12g-7}(\beta)
\end{align*}
where the sum is taken over representatives of 
the $\sim$-equivalence classes of arrows and $\beta$ is any arrow.
Since $\phi^3$ is the identity on $Q_1$, taking other representatives
results in cyclically equivalent terms, hence these potentials are well
defined.

\begin{prop} \label{p:W0W1}
The potentials $W_0$ and $W_1$ are non-degenerate.
The Jacobian algebra $\cP(Q,W_0)$ is infinite-dimensional
whereas $\cP(Q,W_1)$ is finite-dimensional.
\end{prop}
\begin{proof}
By~\cite[\S2.2]{Ladkani12} we know that $W_1$ is the potential associated
by Labardini~\cite{Labardini09} to a triangulation whose adjacency quiver
is $Q$. It consists of two summands; the first is the sum of all $3$-cycles
in $Q$ corresponding to the triangles, and the second is the cycle ``around''
the puncture.
We have shown the finite-dimensionality of its Jacobian algebra
in~\cite{Ladkani12}. The non-degeneracy of $W_1$ follows
by combining the fact that all triangulations do not have self-folded triangles
together with the compatibility between flips and mutations~\cite{Labardini09}.

The potential $W_0$, obtained from $W_1$ by omitting the
cycle around the puncture, was introduced in our previous
work~\cite[\S4.3]{Ladkani11}, where the relevant results are explained.
\end{proof}

We note that a similar statement appears in~\cite[Proposition~9.13]{GLS13}.

\begin{cor} \label{c:nogreen}
$Q$ has no maximal green sequences.
\end{cor}
\begin{proof}
By the previous proposition, $Q$ has a non-degenerate potential whose
Jacobian algebra is infinite-dimensional. The result now follows
from Proposition~8.1 in~\cite{BDP13}.
\end{proof}

Let $\cC$ be the generalized cluster category associated to the quiver
with potential $(Q,W_1)$ whose Jacobian algebra
is finite-dimensional. By the results of Amiot~\cite{Amiot09},
it is a $\Hom$-finite $2$-Calabi-Yau triangulated category with
a canonical cluster-tilting object $\Gamma$ whose endomorphism algebra
is $\cP(Q,W_1)$. We denote by $\Sigma$ the suspension in $\cC$.

We can repeat the argument of Plamondon~\cite[Example~4.3]{Plamondon13}
to deduce that there are cluster-tilting objects in $\cC$ that are not
reachable from $\Gamma$ via finite sequences of mutations.

\begin{prop} \label{p:noreach}
The cluster-tilting object $\Sigma \Gamma$ in $\cC$ is not reachable by a
finite sequence of mutations from the canonical cluster-tilting object
$\Gamma$.
\end{prop}
\begin{proof}
The quiver with potential $(Q,W_1)$ is non-degenerate, hence in any iterated
mutation of quivers with potentials, the underlying quiver is the
iterated quiver mutation of $Q$. If follows from~\cite{KellerYang11} that
the quiver $Q_U$ of the endomorphism algebra of a cluster-tilting object
in $\cC$ obtained from the canonical one by iterated mutation 
does not have oriented cycles of length $2$ and it belongs to the mutation
class of $Q$.

We use the notion of index from~\cite{DehyKeller08} and
write the index of an object $Y$ in $\cC$ 
with respect to such a cluster-tilting object $U$ in $\cC$ as
$\ind_{U} Y = \sum_{i=1}^n y_i [U_i]$ in the
split Grothendieck
group of $\add U$, where $U_1, \dots, U_n$ are the
non-isomorphic indecomposable summands of $U$.
If $U'$ is the cluster-tilting object obtained from $U$ by 
exchanging the summand $U_k$ and
$\ind_{U'} Y = \sum_{i=1}^n y'_i [U'_i]$ is the corresponding index,
then by~\cite{DehyKeller08} the coefficients $y'_i$ are obtained from
$y_i$ according to the mutation rule
\[
y'_i = \begin{cases}
-y_k & \text{if $i=k$,} \\
y_i + r[y_k]_+
& \text{if $i \neq k$ and there are $r$ arrows $i \to k$ in $Q_U$,} \\
y_i - r[-y_k]_+
& \text{if $i \neq k$ and there are $r$ arrows $k \to i$ in $Q_U$}
\end{cases}
\]
where $[y]_+ = \max(0,y)$.
By part~\eqref{it:deg2} of Proposition~\ref{p:combmodel}, the
in-degree and out-degree of any vertex of $Q_U$ are equal to $2$
and hence
\[
y'_1 + \dots + y'_n 
= -y_k + \sum_{i \neq k} y_i + 2 \left([y_k]_+ - [-y_k]_+ \right)
= y_k + \sum_{i \neq k} y_i 
= y_1 + \dots + y_n
\]
so that the sum of the coefficients appearing in the index is
constant for all the cluster-tilting objects $U$ reachable from the
canonical one $\Gamma$.

Now write $\Gamma=\Gamma_1 \oplus \dots \oplus \Gamma_n$. Obviously
\begin{align*}
\ind_{\Gamma} \Sigma \Gamma = \sum_{i=1}^n (-1) \cdot [\Gamma_i]
&&
\ind_{\Sigma \Gamma} \Sigma \Gamma = \sum_{i=1}^n 1 \cdot [\Sigma \Gamma_i]
\end{align*}
so the corresponding sums of coefficients are $-n$ and $n$, respectively.
Hence $\Sigma \Gamma$ is not reachable from $\Gamma$ by a sequence of
mutations.
\end{proof}

\begin{remark}
Corollary~\ref{c:nogreen} could be also be deduced from
Proposition~\ref{p:noreach} as in the proof
of~\cite[Proposition~2.21]{BDP13}.
In fact, Proposition~\ref{p:noreach} implies the stronger statement that
$Q$ does not have a reddening sequence as defined in~\cite{Keller13}.
\end{remark}

\begin{remark}
Consider the graph of (isomorphism classes of basic) cluster-tilting objects
in $\cC$, where edges correspond to mutations.
In the course of the above proof we have seen that this graph has at least
two connected components, one containing $\Gamma$ and another containing
$\Sigma \Gamma$.
It is interesting to compare this with Proposition~7.10 of~\cite{FST08}
asserting that for a once-punctured closed surface, the tagged arc complex
and its dual graph have two connected components.
\end{remark}

\section{On the class $\cP$}
\label{sec:P}

We describe a procedure of inductively building quivers from simpler ones
while keeping various properties regarding the possible potentials on them.

\begin{defn}
A \emph{triangular extension}~\cite[\S3.3]{Amiot09} of two quivers
$Q'$ and $Q''$ is any quiver $Q$ obtained from the
disjoint union of $Q'$ and $Q''$ by adding some new arrows (possibly none)
from vertices of $Q'$ to vertices of $Q''$.
\end{defn}

\begin{defn}
Let $\cQ$ be a set of quivers. Denote by $\langle \cQ \rangle$ the
smallest set of quivers that contains $\cQ$ and is closed under performing
quiver mutations and triangular extensions.
\end{defn}

\begin{defn}
Let $\bullet$ be the quiver with one vertex and no arrows. 
The class $\cP$ of Kontsevich and Soibelman~\cite[\S8.4]{KS08} is
$\langle \bullet \rangle$.
In particular, all quivers that are mutation equivalent to ones
without oriented cycles are in $\cP$.
\end{defn}

\begin{remark} \label{rem:triang}
Let $Q$ be a triangular extension of two quivers $Q'$ and $Q''$.
If $W$ is a potential on $Q$, let $W|_{Q'}$ be the \emph{restriction}
of $W$ to $Q'$ consisting of all the terms of $W$ whose arrows lie
entirely in $Q'$. By abuse of notation we identify the complete path
algebra of $Q'$ with its image in the complete path algebra of $Q$,
so we can think of $W|_{Q'}$ as a potential on $Q'$ or on $Q$
as needed.
We define $W|_{Q''}$ similarly.

We observe the following:
\begin{enumerate}
\renewcommand{\theenumi}{\alph{enumi}}
\item
Any cycle in $Q$ is already a cycle contained in $Q'$ or in $Q''$.

\item
If $W$ is a potential on $Q$, then $W = W|_{Q'} + W|_{Q''}$.
Hence for any arrow $\alpha \in Q_1$,
\[
\partial_\alpha W = \begin{cases}
\partial_\alpha \left(W|_{Q'}\right)  & \text{if $\alpha \in Q'_1$,} \\
\partial_\alpha \left(W|_{Q''}\right) & \text{if $\alpha \in Q''_1$,} \\
0 & \text{otherwise.}
\end{cases}
\]

\item
If $W'$, $W''$ are rigid potentials on $Q'$, $Q''$ respectively,
then $W=W'+W''$ is a rigid potential on $Q$. Indeed, we have to verify
that any cycle in $Q$ is cyclically equivalent to an element in the
Jacobian ideal of $W$ and this follows from the previous claims
since $W'=W|_{Q'}$ and $W''=W|_{Q''}$.
\end{enumerate}
\end{remark}

\begin{prop}
Let $(\mathrm{P})$ be any of the following properties of a quiver $Q$:
\begin{enumerate}
\renewcommand{\theenumi}{P\arabic{enumi}}
\item \label{it:P:fd}
Any non-degenerate potential on $Q$ has finite-dimensional Jacobian algebra;

\item \label{it:P:rigid}
Any non-degenerate potential on $Q$ is rigid;

\item \label{it:P:uniq}
A non-degenerate potential on $Q$ is unique up to right equivalence;

\item \label{it:P:exrigid}
There is a rigid potential on $Q$;
\end{enumerate}
and let $\cQ$ be a set of quivers.
If each quiver in $\cQ$ has property $(\mathrm{P})$, then all quivers in
$\langle \cQ \rangle$ have property $(\mathrm{P})$.
\end{prop}
\begin{proof}
We need to verify that each of the properties is preserved under
quiver mutations and triangular extensions.

First, we assume property $(\mathrm{P})$ for a quiver $Q$ and show it for
a mutation $\mu_k(Q)$. Let $W$ a non-degenerate potential on $\mu_k(Q)$.
Then the corresponding mutation of the quiver with potential $(\mu_k(Q),W)$
is $(Q, W')$ for some non-degenerate potential $W'$ on $Q$.

The following claims can be found in~\cite{DWZ08}.
\begin{enumerate}
\item[\protect{\eqref{it:P:fd}}]
The Jacobian algebra of $(\mu_k(Q),W)$ is finite-dimensional if that of
$(Q,W')$ is.

\item[\protect{\eqref{it:P:rigid}}]
$(\mu_k(Q),W)$ is rigid if $(Q,W')$ is.

\item[\protect{\eqref{it:P:uniq}}]
If $W_1, W_2$ are two non-degenerate potentials on $\mu_k(Q)$, then
the corresponding potentials $W'_1, W'_2$ on $Q$ are right equivalent
if and only if $W_1, W_2$ are.

\item[\protect{\eqref{it:P:exrigid}}]
If $(Q,W)$ is rigid, then its mutation is rigid and is of
the form $(\mu_k(Q),W')$ since rigid potentials are non-degenerate.
\end{enumerate}

Now let $Q$ be a triangular extension of $Q'$ and $Q''$ and assume that
property~$(\mathrm{P})$ holds for $Q'$ and $Q''$.
We show that it also holds for $Q$.
Let $W$ be a potential on $Q$. Then by Remark~\ref{rem:triang},
$W = W'+W''$ with $W'=W|_{Q'}$, $W''=W|_{Q''}$.
If $W$ is non-degenerate, then the potentials $W'$, $W''$ are also
non-degenerate~\cite[Corollary 22]{Labardini09}.

\begin{enumerate}
\item[\protect{\eqref{it:P:fd}}]
Follows from~\cite[Prop.~3.7]{Amiot09}.

\item[\protect{\eqref{it:P:rigid}}]
Follows from Remark~\ref{rem:triang}.

\item[\protect{\eqref{it:P:uniq}}]
If $W'_1, W'_2$ are right equivalent potentials on $Q'$
and $W''_1, W''_2$ are right equivalent potentials on $Q''$,
then $W'_1 + W''_1$ and $W'_2 + W''_2$ are right equivalent
potentials on $Q$. The right equivalence is obtained by
``gluing'' the right equivalences on $Q'$ and $Q''$,
defining it to be the identity on all the arrows from $Q'$ to $Q''$.

\item[\protect{\eqref{it:P:exrigid}}]
Follows from Remark~\ref{rem:triang}.
\end{enumerate}
\end{proof}

\begin{theorem} \label{t:P}
Any quiver in class $\cP$ has a 
unique non-degenerate potential (up to right equivalence) which is rigid
and its Jacobian algebra is finite-dimensional.
\end{theorem}
\begin{proof}
The quiver $\bullet$ trivially satisfies the properties
\eqref{it:P:fd}, \eqref{it:P:uniq} and~\eqref{it:P:exrigid}, since the only
potential is zero.
\end{proof}

\begin{defn}
Define a partial order $<$ on the set of quivers as follows:
\[
\text{$Q' < Q$} \quad
\text{if $|Q'_0| < |Q_0|$, or $|Q'_0| = |Q_0|$ and $|Q'_1| < |Q_1|$}.
\]
For a quiver $Q$, let $Q^{<}$ be the set of all quivers smaller than $Q$
in the order just defined.
The quivers in $Q^{<}$ can be thought to be simpler than $Q$, as they
either have less vertices or the same number of vertices but less arrows.
\end{defn}

The next proposition shows that adjacency quivers of triangulations of
once-punctured closed surfaces cannot be built from simpler quivers
by using the operations defining the class $\cP$.

\begin{prop} \label{p:Qminus}
Let $Q$ be an adjacency quiver of a triangulation of a once-punctured closed
surface. Then $Q \not \in \langle Q^{<} \rangle$.
\end{prop}
\begin{proof}
$Q$ is not mutation equivalent to any quiver in $Q^{<}$ since all members in
the mutation class of $Q$ have the same number of arrows~\cite{Ladkani11}.
Moreover, any mutation of
$Q$ is not a triangular extension of any two smaller quivers since it has
an Eulerian cycle (cf.\ Prop.~\ref{p:combmodel})
and hence there exists a path between any two vertices.
\end{proof}

As a corollary we obtain part~\eqref{it:T:P} of the theorem.
\begin{cor} \label{c:Qg0noP}
An adjacency quiver of a triangulation of a once-punctured closed surface
does not belong to class $\cP$.
\end{cor}

\begin{remark}
We could also deduce this corollary by combining Theorem~\ref{t:P} with
Proposition~\ref{p:W0W1}.
\end{remark}

There is no general procedure to determine if a given quiver belongs to the
class $\cP$. However, if the mutation class of the quiver is finite, a naive
algorithm would be to enumerate on the members of that class, search for
quivers which are triangular extensions of two smaller ones and
then apply the algorithm recursively for the smaller quivers.
The (connected) quivers whose mutation class is finite were classified
by Felikson, Shapiro and Tumarkin~\cite{FST12}. 
Apart from quivers with two vertices and some $r \geq 3$ arrows
from one vertex to the other, such quivers
either arise as adjacency
quivers of triangulations of marked surfaces, or they belong to $11$
exceptional mutation classes.
The next theorem shows that most of the quivers with
finite mutation class belong to the class $\cP$.

\begin{theorem} \label{t:Pfinite}
A connected quiver whose mutation class is finite belongs to the class $\cP$
if and only if it is \textbf{not} one of the following:
\begin{enumerate}
\renewcommand{\theenumi}{\alph{enumi}}
\item \label{it:T:closed}
An adjacency quiver of a triangulation of a closed surface; or

\item \label{it:T:Qg1}
The quiver $Q_{1,1}$ shown in Figure~\ref{fig:torus}; or

\item \label{it:T:X7}
A member of the mutation class of the quiver $X_7$.
\end{enumerate}
\end{theorem}

Applying Theorem~\ref{t:P}, we deduce the following.
\begin{cor} \label{c:Pfinite}
A quiver whose mutation class is finite and none of its connected
components belong to the above
families \eqref{it:T:closed}, \eqref{it:T:Qg1} or~\eqref{it:T:X7}
has a unique non-degenerate potential (up to right equivalence) which is
rigid and its Jacobian algebra is finite-dimensional.
\end{cor}

\begin{remark}
The uniqueness of potentials for these quivers
has also been shown by Geiss, Labardini and Schr\"{o}er
in~\cite{GLS13} using other techniques.
In that work the authors also show
the uniqueness, up to weak right equivalence, of potentials 
for many mutation classes in the family~\eqref{it:T:closed}
and construct two inequivalent potentials on $Q_{1,1}$.
\end{remark}

\begin{remark}
Corollary~\ref{c:Pfinite} does not give the potential explicitly.
An explicit construction for adjacency quivers of triangulations
has been given by Labardini~\cite{Labardini09}. For the exceptional
mutation classes that are not acyclic,
potentials can be found in our work~\cite{Ladkani11a}.
\end{remark}

The rest of this section is devoted to the proof of Theorem~\ref{t:Pfinite}.
We start by recording the following observation.
\begin{lemma} \label{l:onepoint}
If $Q$ has a sink $i$ and $Q \setminus \{i\}$ belongs to $\cP$, then
$Q$ belongs to $\cP$.
\end{lemma}

Checking which of the $11$ exceptional finite mutation classes belongs to
$\cP$ is routine using Lemma~\ref{l:onepoint}
(or can be done on a computer). In particular,
the quivers $E_n$, $\wt{E}_n$, $E_n^{1,1}$ 
for $n=6,7,8$ and $X_6$ belong to $\cP$, but $X_7$ does not
(the latter two quivers were introduced in~\cite{DerksenOwen08}).

For adjacency quivers of triangulations with at least two marked points,
we adapt the arguments of Muller in~\cite[\S10]{Muller13}.
Denote by $\cS$ the set of marked surfaces $(S,M)$ which are not closed
and have at least two marked points.

\begin{lemma} \label{l:cut}
Let $(S,M) \in \cS$ and assume that it is not an unpunctured disc.
Then there exists a triangulation of $(S,M)$ whose adjacency quiver $Q$
has a sink $i$ and moreover the quiver $Q \setminus \{i\}$ is the
adjacency quiver of triangulation of a marked surface $(S',M') \in \cS$.
\end{lemma}
\begin{proof}
We follow the proof of Theorem~10.6 in~\cite{Muller13}. There are
three reduction cases. In each case we find 
an arc $i$ and a triangulation of $(S,M)$ containing it
such that in its adjacency quiver $Q$ the corresponding
vertex $i$ is a sink and
$Q \setminus \{i\}$ is the adjacency quiver of
a triangulation of the marked surface $(S',M')$ obtained by
cutting along the arc $i$, which still belongs to $\cS$.
We demonstrate this in Figure~\ref{fig:reduce}.

If $(S,M)$ has a puncture $p$, then there are arcs $i, j, k$ 
as in picture~(1a)
(it is possible that for the marked points on the boundary, $q=q'$).
The surface $(S',M')$ is shown in picture~(1b).

If $(S,M)$ has at least two boundary components, then there is an arc $i$
connecting marked points on distinct boundary components. Find arcs $j, k$
as in picture~(2a) (it is possible that these arcs coincide).
The surface $(S',M')$ is shown in picture~(2b).

If $(S,M)$ has one boundary component with at least two marked points
and the genus of $S$ is not zero, then there is some arc $i$ which connects
distinct marked points on the boundary component of $S$ such that cutting
along $i$ does not disconnect $S$. We proceed as in the previous case.
\end{proof}

\begin{figure}
\[
\begin{array}{ccc}
\begin{array}{c}
\xymatrix@=0.5pc{
\\
&&&&& {\times_p} \ar@{-}[drrr]^(0.4){i} \\
{} \ar@{-}@<0.2ex>[rr] \ar@{.}@<-0.2ex>[rr]
&& {\cdot_{q'}} \ar@{-}@/^3pc/[rrrrrr]^j
\ar@{-}@<0.2ex>[rrrrrr] \ar@{.}@<-0.2ex>[rrrrrr] \ar@{-}[urrr]^k
&&&&&& {\cdot_q} \ar@{-}@<0.2ex>[rr] \ar@{.}@<-0.2ex>[rr] && {}
}
\end{array}
& \longrightarrow &
\begin{array}{c}
\xymatrix@=0.5pc{
\\
&&&&& {\cdot_{p_1}} \ar@{-}@<0.2ex>[drrr] \ar@{.}@<-0.2ex>[drrr]
\ar@{.}@<0.2ex>[dr] \ar@{-}@<-0.2ex>[dr] \\
{} \ar@{-}@<0.2ex>[rr] \ar@{.}@<-0.2ex>[rr]
&& {\cdot_{q'}} \ar@{-}@/^3pc/[rrrrrr]^j
\ar@{-}@<0.2ex>[rrrr] \ar@{.}@<-0.2ex>[rrrr] \ar@{-}[urrr]^k
&&&& {\cdot_{p_2}}
&& {\cdot_q} \ar@{-}@<0.2ex>[rr] \ar@{.}@<-0.2ex>[rr] && {}
}
\end{array}
\\
(1a) & & (1b) \\ \\
\begin{array}{c}
\xymatrix@R=1pc@C=0.5pc{
{} \ar@{.}@<0.2ex>[rr] \ar@{-}@<-0.2ex>[rr] 
&& {\cdot_q} \ar@{.}@<0.2ex>[rrrr] \ar@{-}@<-0.2ex>[rrrr]
\ar@{-}[dd]_j \ar@{-}[ddrrrr]^i
&&&& {\cdot_{q'}} \ar@{.}@<0.2ex>[rr] \ar@{-}@<-0.2ex>[rr]
\ar@{-}[dd]^k
&& {} \\ \\
{} \ar@{-}@<0.2ex>[rr] \ar@{.}@<-0.2ex>[rr] 
&& {\cdot_{p'}} \ar@{-}@<0.2ex>[rrrr] \ar@{.}@<-0.2ex>[rrrr]
&&&& {\cdot_p} \ar@{-}@<0.2ex>[rr] \ar@{.}@<-0.2ex>[rr]
&& {}
}
\end{array}
& \longrightarrow &
\begin{array}{c}
\xymatrix@R=1pc@C=0.5pc{
{} \ar@{.}@<0.2ex>[rr] \ar@{-}@<-0.2ex>[rr] 
&& {\cdot_{q_1}} \ar@{.}@<0.2ex>[ddrrr] \ar@{-}@<-0.2ex>[ddrrr]
\ar@{-}[dd]_j
& {\cdot_{q_2}} \ar@{-}@<0.2ex>[ddrrr] \ar@{.}@<-0.2ex>[ddrrr]
\ar@{.}@<0.2ex>[rrr] \ar@{-}@<-0.2ex>[rrr]
&&& {\cdot_{q'}} \ar@{.}@<0.2ex>[rr] \ar@{-}@<-0.2ex>[rr]
\ar@{-}[dd]^k
&& {} \\ \\
{} \ar@{-}@<0.2ex>[rr] \ar@{.}@<-0.2ex>[rr] 
&& {\cdot_{p'}} \ar@{-}@<0.2ex>[rrr] \ar@{.}@<-0.2ex>[rrr]
&&& {\cdot_{p_1}}
& {\cdot_{p_2}} \ar@{-}@<0.2ex>[rr] \ar@{.}@<-0.2ex>[rr]
&& {}
}
\end{array}
\\
(2a) && (2b)
\end{array}
\]
\caption{Reduction by cutting along the arc $i$
at a puncture $p$ (top row) and at two distinct boundary points
$p,q$ (bottom row).}
\label{fig:reduce}
\end{figure}
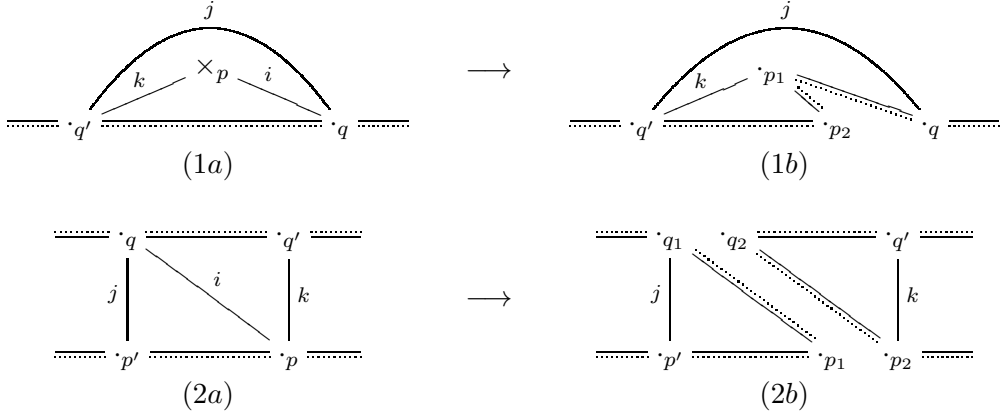

\begin{lemma} \label{l:SinP}
An adjacency quiver of a marked surface in $\cS$ belongs to the class $\cP$.
\end{lemma}
\begin{proof}
If $(S,M)$ is an unpunctured disc, an
adjacency quiver is mutation equivalent to
a Dynkin quiver of type $A_n$, hence belongs to $\cP$.
Otherwise, we proceed by induction on the number of arcs
in a triangulation of $(S,M)$, the induction step being an
application of Lemma~\ref{l:cut} and Lemma~\ref{l:onepoint}.
\end{proof}

In order to complete the proof of the ``if'' part of Theorem~\ref{t:Pfinite},
it remains to consider surfaces of genus $g \geq 1$ with exactly
one boundary component and one marked point on that component.
We denote the mutation class of the corresponding adjacency quivers
by $\cQ_{g,1}$, in agreement with the notation in our
papers~\cite{Ladkani11a,Ladkani11}.

\begin{lemma} \label{l:PQg1}
A quiver in $\cQ_{g,1}$ belongs to the class $\cP$ if and only if $g>1$.
\end{lemma}
\begin{proof}
The quiver $Q_{1,1}$ is not a triangular extension and its mutation
class consists of a single element, hence it is not in $\cP$.
From now on assume $g>1$.

Let $(S',\{p'\})$ be a surface of genus $g-1$ with one boundary component
$\gamma'$ and a marked point $p'$ on $\gamma'$. Let $(S'',\{p''\})$ be a
torus with one boundary component $\gamma''$ and a marked point $p''$
on $\gamma''$.
Gluing these surfaces along $\gamma'$ and $\gamma''$, identifying
the marked point $p'$ with $p''$, we get a closed surface $(S,\{p\})$ of
genus $g$ with one puncture $p$ which is the image of $p'$ (and $p''$).

Two triangulations $T'$ of $(S',\{p'\})$ and $T''$ of $(S'',\{p''\})$
yield a triangulation $T$ of $(S,\{p\})$ by taking the arcs of $T'$
and $T''$ together with the arc $\gamma$ which is the image in $S$
of $\gamma'$ (or $\gamma''$). Hence the adjacency quiver $Q_T$ is
obtained from the disjoint union of the
\emph{extended} adjacency quivers (cf. Section~\ref{sec:UA}) 
$Q' = \wt{Q}_{T'}$ and $Q'' = \wt{Q}_{T''}$ by identifying the two
frozen vertices corresponding to the boundary segments $\gamma_1$
in $Q'$ and $\gamma_2$ in $Q''$.

The arc $\gamma$ is contained in two triangles of $T$ as in
Figure~\ref{fig:Qg0toQg1}, one arising from $S'$ and the other from
$S''$. By cutting out a disc containing $p$
within one of the triangles and adding an arc ``parallel'' to $\gamma$
we get a triangulation of a surface of genus $g$ with one boundary
component and one marked point. Denoting the parallel arcs by $\gamma'$
and $\gamma''$, the adjacency quiver $Q$ of this triangulation is
obtained from the disjoint union of $Q'$ and $Q''$ by adding a single
arrow $\gamma' \to \gamma''$, hence it is a triangular extension
of $Q'$ and~$Q''$.

The argument in the proof of Lemma~\ref{l:flipext} shows that
the extended adjacency quivers $Q'$ and $Q''$ are (usual)
adjacency quivers for surfaces with two marked points, hence
by Lemma~\ref{l:SinP} they belong to $\cP$.
It follows that $Q \in \cP$ as well.
\end{proof}

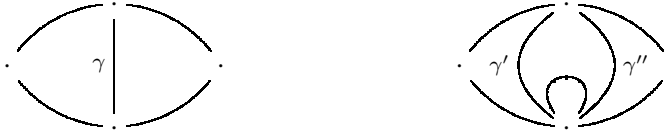
\begin{figure}
\begin{align*}
\xymatrix@=1pc{
&& {\cdot} \ar@{-}@/^/[drr] \\
{\cdot} \ar@{-}@/^/[urr] && && {\cdot} \ar@{-}@/^/[dll] \\
&& {\cdot} \ar@{-}@/^/[ull] \ar@{-}[uu]^{\gamma}
}
&&
\xymatrix@=1pc{
&& {\cdot} \ar@{-}@/^/[drr] \\
{\cdot} \ar@{-}@/^/[urr] && && {\cdot} \ar@{-}@/^/[dll] \\
&& {\cdot} \ar@{-}@/^/[ull]
\ar@{-}@/^1.5pc/[uu]^{\gamma'} \ar@{-}@/_1.5pc/[uu]_{\gamma''}
\ar@{-}@(ur,ul)[]
\ar@{.}@<0.2ex>@(ur,ul)[]
}
\end{align*}
\caption{Inserting a boundary component.}
\label{fig:Qg0toQg1}
\end{figure}

We complete the proof of the ``only if'' part of Theorem~\ref{t:Pfinite}
by considering closed surfaces and applying an algebraic argument generalizing
Corollary~\ref{c:Qg0noP} to arbitrary number of punctures.

\begin{lemma}
An adjacency quiver of a triangulation of a closed surface does not belong
to the class $\cP$.
\end{lemma}
\begin{proof}
If the surface is not a sphere with $4$ or $5$ punctures, 
the potential associated to the triangulation by Labardini~\cite{Labardini09}
has recently been shown by him to be non-degenerate~\cite{Labardini12}.
However, by~\cite{Ladkani12} it is not rigid
and the claim follows from Theorem~\ref{t:P}.

The mutation classes of adjacency quivers of triangulations of
a sphere with $4$ or $5$ punctures consist of $4$ and $26$ quivers,
respectively. 
One checks on a computer that for each of these quivers,
any two vertices $i$ and $j$ are connected by a path from $i$ to $j$
as well as a path from $j$ to $i$.
Therefore these quivers cannot be triangular
extensions.
Since any quiver in the mutation class is not a triangular extension,
we get the claim.
\end{proof}

We conclude with a few remarks.
\begin{remark}
By combining Lemma~\ref{l:PQg1} with~\cite[Theorem~10.5]{Muller13} we see
that there are quivers in class $\cP$ whose cluster algebra
is not locally acyclic as defined in~\cite{Muller13}.
Hence the two notions do not coincide.
\end{remark}

\begin{remark}
A counting argument shows that each quiver in $\cQ_{g,1}$ does not have
any sinks or sources.
Indeed, the number of arrows is one less than twice the number of
vertices and the in-degree and out-degree of any vertex are bounded by $2$.
From Lemma~\ref{l:PQg1} we deduce that there are quivers in $\cP$ that are
not mutation equivalent to quivers with a sink or a source.

This could be made more formal as follows.
Consider the class $\cP'$ of quivers defined similarly to the class $\cP$,
except that we only allow triangular extensions of two quivers where
one of them is a point (in analogy with one-point extensions and
co-extensions). By construction, any quiver in $\cP'$ is mutation
equivalent to a quiver with a sink or a source.
We therefore get a sequence of strict inclusions
\[
\left\{\text{quivers that are mutation-equivalent to acyclic ones} \right\}
\subsetneqq \cP' \subsetneqq \cP .
\]
The proof of Theorem~\ref{t:Pfinite} shows that a connected quiver whose
mutation class is finite belongs to the class $\cP'$ if and only if
it belongs to $\cP$ and is not a member of a class $\cQ_{g,1}$ for some
$g>1$.
\end{remark}

\section{Explicit construction of some quivers}
\label{sec:constr}

In order to make our results more concrete, we present explicit
constructions of quivers appearing in the main theorem.
Such constructions were already presented in our previous
work~\cite[\S3.2]{Ladkani11} yielding quivers having some
block structure but with double arrows.
Here we present another procedure yielding quivers without double arrows.

Recall that a closed surface of genus $g \geq 1$ can be obtained by taking
the fundamental polygon with $4g$ sides labeled $1, 2, 1, 2, \dots,
2g-1, 2g, 2g-1, 2g$ and identifying sides having the same label (with
appropriate orientations that will not be relevant here).
Under this identification, all the vertices of the polygon are being mapped
to the same point.
Thus, any triangulation of this $4g$-gon gives rise to a triangulation of
a once-punctured closed surface of genus $g$ by identifying the puncture
with that common point and adding the $2g$ arcs corresponding to the
distinct sides of the $4g$-gon.

From now on assume that $g \geq 2$. First we add $2g$ new arcs labeled
$2g+1,2g+2,\dots,4g$ such that each arc $2g+i$ is the side of a triangle
as shown below:
\begin{align} \label{e:triang}
\begin{array}{c}
\xymatrix@=1pc{
{\cdot} \ar@{-}[dd]_{2g+i} \ar@{-}[drr]^{i+1} \\
&& {\cdot} \ar@{-}[dll]^i \\
{\cdot}
}
\end{array}
\text{($1 \leq i < 2g$ odd)}
&&
\begin{array}{c}
\xymatrix@=1pc{
{\cdot} \ar@{-}[dd]_{2g+i} \ar@{-}[drr]^i \\
&& {\cdot} \ar@{-}[dll]^{i+1} \\
{\cdot}
}
\end{array}
\text{($1 < i \leq 2g$ even)}
\end{align}
These new arcs encircle a $2g$-gon inside the fundamental $4g$-gon,
and any triangulation of this inner $2g$-gon (consisting of additional
$2g-3$ arcs) yields a triangulation of the surface with $6g-3$ arcs.
Its adjacency quiver will not contain double arrows, since
by our choice of triangles in~\eqref{e:triang}, for any two arcs there is
at most one triangle having both of them as sides.

In particular, we can choose a triangulation of the inner $2g$-gon whose
adjacency quiver is a linearly oriented Dynkin quiver $A_{2g-3}$ and
get a triangulation of the once-punctured surface of genus $g \geq 2$ whose
adjacency quiver has
$6g-3$ vertices numbered $1,2,\dots,6g-3$ with the arrows
\begin{align*}
i &\to 2g+i && i \to 2g+(i-1) && 2g+i \to i+1
&& \text{($1 \leq i < 2g$ odd)} \\
i &\to i-1 && i \to i+1 && 2g+i \to i
&& \text{($1 < i \leq 2g$ even)}
\end{align*}
(here $i-1$ and $i+1$ are computed ``modulo $2g$'' to take values
in the range $[1,2g]$, i.e.\ if $i=1$ then $i-1$ is $2g$ and
if $i=2g$ then $i+1$ is $1$)
together with the arrows
\begin{align*}
2g+2 &\to 4g+1 && 4g+1 \to 2g+3 && 4g+1 \to 2g+1 \\
2g+i+1 &\to 4g+i && 4g+i \to 2g+i+2  && 4g+i \to 4g+i-1
&& (2 \leq i \leq 2g-3) \\
4g-1 &\to 4g && 2g+1 \to 2g+2 && 4g \to 6g-3
\end{align*}
corresponding to the chosen triangulation of the inner $2g$-gon.

Examples of these triangulations for surfaces of genus $2$ and $3$
are shown in Figure~\ref{fig:Tg0}. The corresponding adjacency
quivers are those appearing in Figure~\ref{fig:Qg0}.

\begin{figure}
\begin{align*}
\begin{array}{c}
\xymatrix@=0.3pc{
& && {\cdot}
\ar@{-}[drr]^1 \\
& {\cdot} \ar@{-}[urr]^4 \ar@{-}@/_/[rrrr]^8
&& && {\cdot} \ar@{-}[ddr]^2
\ar@{-}@/_/[dddd]^5 \ar@{-}[ddddllll]^9 \\ \\
{\cdot} \ar@{-}[uur]^3 && && && {\cdot} \ar@{-}[ddl]^1 \\ \\
& {\cdot} \ar@{-}[uul]^4  \ar@{-}@/_/[uuuu]^7
&& && {\cdot} \ar@{-}[dll]^2 \ar@{-}@/_/[llll]^6 \\
& && {\cdot} \ar@{-}[ull]^3
}
\end{array}
& &
\begin{array}{c}
\xymatrix@=0.2pc{
& && && {\cdot} \\
& && {\cdot} \ar@{-}[urr]^6 \ar@{-}@/_/[rrrr]^{12}
&& && {\cdot} \ar@{-}[ull]_1 \ar@{-}@/_/[ddddrrr]_7 
\ar@{-}@/_/[dddddddd]^{13} \ar@{-}[ddddddddllll]_(0.6){14}
\ar@{-}@/^/[ddddlllllll]^{15} \\ \\
& {\cdot} \ar@{-}[uurr]^5 && && && && {\cdot} \ar@{-}[uull]_2 \\ \\
{\cdot} \ar@{-}[uur]^6 \ar@{-}@/_/[uuuurrr]_(0.7){11}
& && && && && & {\cdot} \ar@{-}[uul]_1
\ar@{-}@/_/[ddddlll]_8 \\ \\
& {\cdot} \ar@{-}[uul]^5 && && && && {\cdot} \ar@{-}[uur]_2 \\ \\
& && {\cdot} \ar@{-}[uull]^4 \ar@{-}@/_/[uuuulll]_{10}
&& && {\cdot} \ar@{-}[uurr]_3 \ar@{-}@/_/[llll]_9 \\
& && && {\cdot} \ar@{-}[ull]^3 \ar@{-}[urr]_4
}
\end{array}
\end{align*}
\caption{Triangulations of closed surfaces of genus $g$ with one
puncture, for $g=2,3$. Arcs having the same label are identified.}
\label{fig:Tg0}
\end{figure}

\bibliographystyle{amsplain}
\bibliography{punct}

\end{document}